\documentclass[a4paper,12pt]{article}
\usepackage{fullpage} 
\usepackage[latin2]{inputenc}
\usepackage[magyar,english]{babel}

\usepackage{tikz}

\usepackage{amsmath}
\usepackage{amsthm}
\usepackage{amssymb}
\usepackage{amsrefs}

\newtheorem{thm}{Theorem}
\newtheorem{conj}{Conjecture}

\newtheorem{lem}{Lemma}
\newtheorem{defi}{Definition}
\newtheorem{question}{Question}
\newtheorem{claim}{Claim}
\newtheorem{corr}{Corollary}

\theoremstyle{remark}
\newtheorem*{remark}{Remark}
\providecommand{\keywords}[1]{\textbf{\textit{Keywords ---}} #1}

\begin{document} 

\title{New bounds on Simonyi's conjecture}

\author{Daniel Soltész \thanks{Research partially supported by the Hungarian Foundation for Scientific Research Grant (OTKA) No. 108947} \\ solteszd@math.bme.hu \\ Department of Computer Science and Information Theory, \\ Budapest University of Technology and Economics}

\maketitle

\begin{abstract}

We say that a pair $(\mathcal{A},\mathcal{B})$ is a recovering pair if $\mathcal{A}$ and $\mathcal{B}$ are set systems on an $n$ element ground set, such that for every $A,A' \in \mathcal{A}$ and $B,B' \in \mathcal{B}$ we have that ($A \setminus B = A' \setminus B'$ implies $A=A'$) and symmetrically ($B \setminus A = B' \setminus A'$ implies $B=B'$). G. Simonyi conjectured that if $(\mathcal{A},\mathcal{B})$ is a recovering pair, then $|\mathcal{A}||\mathcal{B}|\leq 2^n$. For the quantity $|\mathcal{A}||\mathcal{B}|$ the best known upper bound is $2.3264^n$ due to Körner and Holzman. In this paper we improve this upper bound to $2.284^n$. Our proof is combinatorial.

\end{abstract}


\section{Introduction}

The main subject of this paper is Simonyi's conjecture.
\begin{conj} [Simonyi] \cite{ahlswedesimonyi} \label{Simonyi}
Let $\mathcal{A}$ and $\mathcal{B}$ be set systems on an $n$ element ground set. If for every $A,A' \in \mathcal{A}$ and $B,B' \in \mathcal{B}$ we have that 
$$A \setminus B = A' \setminus B' \Rightarrow A=A'$$ 
$$B \setminus A = B' \setminus A' \Rightarrow B=B'$$
then $|\mathcal{A}||\mathcal{B}| \leq 2^n.$
\end{conj}
A pair of set systems $(\mathcal{A},\mathcal{B})$ that satisfy the conditions of the conjecture is called a recovering pair. We define the size of a pair $(\mathcal{A},\mathcal{B})$ to be the quantity $|\mathcal{A}||\mathcal{B}|$. If Conjecture \ref{Simonyi} is true, it is best possible as one can take an arbitrary set $C \subseteq [n]$, and let $\mathcal{A}$ be every set contained in $C$ and $\mathcal{B}$ be every set contained in the complement of $C$. The best known upper bound on the size of a recovering pair is $2.3264^n$ due to Körner and Holzman \cite{kornerholzman}, for the proof they use the subadditivity of the entropy function. Their proof assumes only that 
$$A \setminus B = A' \setminus B \Rightarrow A=A'$$ 
$$B \setminus A = B' \setminus A \Rightarrow B=B'.$$
A pair of set systems satisfying these conditions is called cancellative. There are cancellative pairs of size larger than $2^n$, see \cite{kornerholzman}. Another weakening of the condition is if we only assume that for the pair $(\mathcal{A},\mathcal{B})$ we have that $A \setminus B = A' \setminus B' \Rightarrow A=A'$, but we do not assume the dual condition. The size of such a 'half recovering' pair can also be larger than $2^n$ see \cite{ahlswedesimonyi}. Conjecture \ref{Simonyi} was verified up to $n=8$ in \cite{dustin}. There is a lattice version of Conjecture \ref{Simonyi}, which roughly asserts that if instead of the boolean lattice we consider a lattice that is the product of chains, a similar construction is optimal. For details see \cite{ahlswedesimonyi}, one can also find results about the lattice version in  \cite{simonyisali} \cite{related} \cite{csakany}. There was an unpublished question of Aharoni, that instead of the size of the recovering pair,  $ |\mathcal{A}|| \mathcal{B}|=\sum_{\substack{A_i \in \mathcal{A} \\ B_j \in \mathcal{B}}}1$, can the  quantity $\sum_{\substack{A_i \in \mathcal{A} \\ B_j \in \mathcal{B}}} 2^{|A_i \cap B_j |}$ be also bounded by $2^n$. This was not conjectured, it was more of an invitation to produce a counterexample \cite{private}. In this paper we present such a counterexample, see Corollary \ref{counteraharoni}. In the case when for a fixed $k$ we have that $|A_i|=|B_j|=k$ for every  $ A_i \in \mathcal{A}$ and $B_j \in \mathcal{B}$ and $n$ is sufficiently large, Simonyi and Sali proved that Conjecture \ref{Simonyi} is true, see \cite{simonyisali}. Their result is very close to the general case, we will show that the case where $|A_i|=|B_j|=c*n$ is equivalent to the general case.

This paper is organized as follows. In the first section we present a new combinatorial approach. This is not enough to improve the Körner Holzman bound, but it is very short, it significantly improves the trivial $3^n$ bound. Here we present the example that answers Aharoni's question. In the second section we fine tune our approach and introduce a second upper bound. In the third section we show that the two bounds combined yield an improvement of $2.284^n.$

\section{Preliminaries and a new proof for a weaker bound}

Let us first present the easy upper bound of $3^n$, and some motivation for Aharoni's question. 
\begin{claim}If $(\mathcal{A},\mathcal{B})$ is a recovering pair, then $$|\mathcal{A}||\mathcal{B}| \leq 3^n.$$
\end{claim}
\begin{proof}The pairs $(A_i \setminus B_j , B_j \setminus A_i)$ are different since we can recover $A_i$ from $A_i \setminus B_j$, and $B_j$ from $B_j \setminus A_i$. But there can be at most $3^n$ pairs of disjoint sets from $[n]$. 
\end{proof} 

With a slight modification of the above proof, we can prove more.

\begin{claim} [Aharoni \cite{private}]If $(\mathcal{A},\mathcal{B})$ is a recovering pair, then 
$$|\mathcal{A}||\mathcal{B}| =  \sum_{\substack{A_i \in \mathcal{A} \\ B_j \in \mathcal{B}}}1 \leq \sum_{\substack{A_i \in \mathcal{A} \\ B_j \in \mathcal{B}}}2^{|A_i \cap B_j |} \leq 3^n.$$
\end{claim}
\begin{proof} The equality and the first inequality is trivial. The last inequality can be proven as follows. For each $A_i$ and $B_j$ and each subset $S$ of $A_i \cap B_j$, the pairs $(A_i \setminus B_j \cup S , B_j \setminus A_i)$ are different, since we can recover $B_j$ from $B_j \setminus A_i$, then from $A_i \setminus B_j \cup S$ we can subtract $B_j$ and from the result we can recover $A_i$, and if we know both $A_i$ and $B_j$, we can easily recover $S$ too. Thus there are at most $3^n$ such pairs as before and the proof is complete.  
\end{proof}

Thus with a slight refinement of the argument we could bound $ \sum_{\substack{A_i \in \mathcal{A} \\ B_j \in \mathcal{B}}}2^{|A_i \cap B_j |}$ instead of the size of the recovering pair. Note that if $ \sum_{\substack{A_i \in \mathcal{A} \\ B_j \in \mathcal{B}}}2^{|A_i \cap B_j |} \leq 2^n$ would hold, it would immediately follow that the only recovering pairs that have size $2^n$ are the ones mentioned in the introduction. We present our counterexample in the end of this section. For our new approach, we need the following definition.

\begin{defi} Let us denote by $f(n)$ the maximal number of solutions of the equation $A_i \cup B_j=[n]$ where $A_i \in \mathcal{A}$ and $B_j \in \mathcal{B}$ such that the maximum is taken over every recovering pair $(\mathcal{A},\mathcal{B})$ on a ground set of size $n$.
\end{defi}

\noindent \textbf{Observation}: Note that given a recovering pair, if $A_i \cup B_i = A_j \cup B_j$ then $A_i \neq A_j$ and $B_i \neq B_j$ otherwise we would have $B_i \setminus A_i= B_j \setminus A_i$ or the other way around. So for each $A_i$ there can be at most one $B_j$ such that $A_i \cup B_j=[n]$. Thus one can think of the solutions as disjoint pairs $(A_1,B_1), (A_2,B_2), \ldots , (A_{f(n)},B_{f(n)})$. 

\begin{lem}
$f(n) \leq \sqrt{2}^n$
\end{lem}
\begin{proof}
The existence of a single pair such that $A_1 \cup B_1 = [n]$ is enough to prove this. All the sets in $\mathcal{B}$ must be different on the complement of $A_1$, and similarly all the sets in $\mathcal{A}$ must be different on the complement of $B_1$. So $|\mathcal{A}||\mathcal{B}| \leq 2^{|A_1^c|+|B_1^c|} \leq 2^n$ holds in every system realizing $f(n)$. Since $f(n)$ is realized by disjoint pairs, $f(n) \leq \min \{|\mathcal{A}|,|\mathcal{B}| \} \leq \sqrt{|\mathcal{A}||\mathcal{B}|} \leq \sqrt{2}^n.$  
\end{proof}

\begin{thm} \label{trivial}
If $(\mathcal{A},\mathcal{B})$ is a recovering pair then $|\mathcal{A}||\mathcal{B}| \leq (1+\sqrt{2})^n$.
\end{thm}
\begin{proof} Let us count the $(A_i,B_j)$ pairs according to their unions. 
$$ |\mathcal{A}||\mathcal{B}| = \sum_{\substack{A_i \in \mathcal{A} \\ B_j \in \mathcal{B}}} 1 = \sum_{C \subseteq [n]}  |\{ (A_i,B_j) | A_i \in \mathcal{A}, B_j \in \mathcal{B} , A_i \cup B_j = C \}| \leq $$
$$ \leq  \sum_{k=0}^n \binom{n}{k} f(k) \leq \sum_{k=0}^n \binom{n}{k} \sqrt{2}^k = (1+\sqrt{2})^n. $$
\end{proof}

This proof is the starting point of our results. It relies heavily on the estimate of $f(n)$. Before proving our main result, let us present a still simple proof for a better upper bound on $f(n)$. For the upper bound we need the following lemma. 

\begin{lem} \label{cupcap}
Let $(\mathcal{A},\mathcal{B})$ be a recovering pair and $A_1, A_2 \in \mathcal{A}$ and $B_1 , B_2 \in \mathcal{B}$. If $A_1 \cup B_1= A_2 \cup B_2$ then $A_1 \cap B_1 \neq A_2 \cap B_2$.
\end{lem}
\begin{proof}
Suppose to the contrary that $A_1 \cup B_1= A_2 \cup B_2$ and $A_1 \cap B_1 = A_2 \cap B_2$. We will get a contradiction by observing that $A_1 \setminus B_2 = B_1 \setminus A_2$, this is demonstrated in Figure \ref{vennproof} where the first column contains the elements in $A_1 \setminus B_1$, the second column contains the elements in $A_1 \cap B_1$ and the third one the elements in $B_1 \setminus A_1$. The meaning of the rows is similar. The fact that we do not need a complete Venn diagram with four sets follows from $A_1 \cup B_1= A_2 \cup B_2$.  The emptiness of four of the areas in the diagram follows from $A_1 \cap B_1 = A_2 \cap B_2$.

\begin{figure}
\begin{center}
\begin{tikzpicture}[scale=0.65]
\draw (0,0) -- (3,0) -- (3,3) -- (0,3) -- (0,0);
\draw (1,0) -- (1,3);
\draw (2,0) -- (2,3);
\draw (0,1) -- (3,1);
\draw (0,2) -- (3,2);
\node(A_1) at (0.5,3.5){$A_1$};
\node(cap1) at (1.5,3.5){$\cap$};
\node(B_1) at (2.5,3.5){$B_1$};

\node(A_2) at (3.5,0.5){$A_2$};
\node(cap2) at (3.5,1.5){$\cap$};
\node(B_2) at (3.5,2.5){$B_2$};

\node(ures1) at (0.5,1.5){$\emptyset$};
\node(ures2) at (1.5,0.5){$\emptyset$};
\node(ures3) at (2.5,1.5){$\emptyset$};
\node(ures4) at (1.5,2.5){$\emptyset$};

\node(x) at (0.5,0.5){$x$};
\end{tikzpicture}
\caption{$A_1 \setminus B_2= x = A_2 \setminus B_1$}
\label{vennproof}
\end{center}
\end{figure}

\end{proof}

We will use the following well known estimate of the order of magnitude of binomial coefficients. 

\begin{lem} \cite{csiszar} \label{analysis}
Let $k \in [0,1/2]$ and  $h(x)= -x \log_2(x)-(1-x)\log_2(1-x) $ the binary entropy function, then we have that  $$ \frac{1}{\sqrt{8nk(1-k)}} 2^{h(k)n} \leq \sum_{i=0}^{kn} \binom{n}{i} \leq 2^{h(k)n}. $$ 
\end{lem}

Since $h(x)$ is unimodular, it has two inverses, we will denote the increasing and the decreasing inverse of $h(x)$ by $h_i^{-1}(x)$ and $h_d^{-1}(x)$ respectively. Note that $0 \leq h_i^{-1}(x) \leq 1/2$ and $h_d^{-1}(x)=1-h_i^{-1}(x)$. In all the cases where we will use Lemma \ref{analysis}, we will only use that $ \binom{n}{kn} \sim 2^{h(k)n}$. Now we are ready to improve the upper bound on $f(n)$.

\begin{lem} \label{legjobbfn}
Let $s=s(n)$ be such that $f(n)= 2^{sn}$. Then $0 \leq 1-2s-h^{-1}_i(s)$, in particular $f(n) \leq 1.3685^n$.
\end{lem}
\begin{proof}

By Lemma \ref{cupcap} we know that there are at least $2^{sn}$ different intersections of type $A_i \cap B_i$ where $A_i \cup B_i = [n]$. Let the pair with the largest such intersection be $(A_j,B_j)$, and let $m$ be such that $|A_j \cap B_j| =mn$. Since the sets in $\mathcal{A}$ must be different on the complement of $B_j$ and similarly the sets in  $\mathcal{B}$ must be different on the complement of $A_j$ we have that $$ 2^{2sn} = f(n)^2 \leq |\mathcal{A}|| \mathcal{B}| \leq 2^{|B_j^c|+|A_j^c|}= 2^{A_j \triangle B_j} = 2^{(1-m)n} $$
From this it follows that 
$$m \leq 1-2s. $$
If $1/2 \leq m $ we have that $s \leq 1/4$ and it is easy to check that the statement of the lemma holds as $1-2s-h^{-1}_i(s)$ is monotone decreasing in $s$ and $0<1-1/2-h_i^{-1}(1/4)$. If $m \leq 1/2$ we have that 
$$2^{sn} \leq \sum_{i=0}^{mn} \binom{n}{i} \leq  2^{h(m)n}. $$
From this it follows that 
$$h^{-1}_{i}(s) \leq m. $$
Which implies the inequality
$$ h^{-1}_{i}(s) \leq m \leq 1-2s $$
Where we have equality if $s$ is approximately $0.4525$, thus $2^s$ is smaller than $1.3685$. 
\end{proof}

\begin{corr}
The bound of Theorem \ref{trivial} can be improved from $(1+\sqrt{2})^n \approx 2.4142^n $ to $2.3685^n$.
\end{corr}

Note that this bound is still slightly weaker than that of Körner and Holzman. We present its proof because of its simplicity, and because we feel that the easiest way to improve our results is to provide a better upper bound on $f(n)$. However to demonstrate the limits of this method, we will show that $f(n)$ is exponential in $n$. Before providing a lower bound for $f(n)$, we need the following lemma, which will be heavily used during subsequent proofs.

\begin{lem}[Multiplying Lemma] \label{multiplying} Let $(\mathcal{A}_1,\mathcal{B}_1)$ and $(\mathcal{A}_2,\mathcal{B}_2)$ be recovering pairs on disjoint ground sets of size $n_1$ and $n_2$ respectively. Let $(\mathcal{A}_3,\mathcal{B}_3)$ be a set system on the union of the two ground sets, defined as follows:

$$ \mathcal{A}_3 := \{ A_i  \dot{\cup} A_j | A_i \in \mathcal{A}_1 , A_j \in \mathcal{A}_2 \}  $$
$$ \mathcal{B}_3 := \{ B_i  \dot{\cup} B_j | B_i \in \mathcal{B}_1 , B_j \in \mathcal{B}_2 \}. $$
Then $(\mathcal{A}_3,\mathcal{B}_3)$ is also a recovering pair
\end{lem}
\begin{proof}
Each set from the family $(\mathcal{A}_3,\mathcal{B}_3)$ consists of two parts, one from $(\mathcal{A}_1,\mathcal{B}_1)$ and the other from $(\mathcal{A}_2,\mathcal{B}_2)$. By symmetry it is enough to show that we can recover $A_k$ from $A_k \setminus B_l$ where $A_k \in \mathcal{A}_3$ and $B_l \in \mathcal{B}_3$. From $[n_1] \cap (A_k \setminus B_l)$ we can recover the part of $A_k$ which comes from $\mathcal{A}_1$ since $(\mathcal{A}_1,\mathcal{B}_1)$ is a recovering pair. Similarly from $[n_2] \cap (A_k \setminus B_l)$ we can recover the part that comes from $\mathcal{A}_2$ since $(\mathcal{A}_2,\mathcal{B}_2)$ is a recovering pair. 
\end{proof}

\begin{remark}
We have that $| \mathcal{A}_3||\mathcal{B}_3|= |\mathcal{A}_1||\mathcal{B}_1||\mathcal{A}_2||\mathcal{B}_2|$. It is also true that if there are exactly $f(n_1)$ solutions of the equation $A_i \cup B_j = [n_1]$ in $(\mathcal{A}_1,\mathcal{B}_1)$, and exactly $f(n_2)$ solutions of $A_i \cup B_j = [n_2]$ in $(\mathcal{A}_2,\mathcal{B}_2)$, then there are exactly $f(n_1)f(n_2)$ solutions of $A_i \cup B_j = [n_1+n_2]$ in $(\mathcal{A}_3,\mathcal{B}_3)$.
\end{remark}

\begin{claim}
$f(6n) \geq 3^{n} \approx 1.2009^{6n}$
\end{claim}
\begin{proof}

It is enough to show that $3 \leq f(6)$ since multiplying such a pair with itself we get the desired bound. Let us define the recovering pair $(\mathcal{A}_{6},{\mathcal{B}_6})$ as follows $$ \mathcal{A}_6 := \{ \{1,2\}^c, \{3,4\}^c, \{5,6\}^c \} \quad \mathcal{B}_6 := \{ \{2,3\}^c, \{4,5\}^c, \{6,1\}^c \} $$ It is left to the reader to verify that this is indeed a recovering pair, and that there are three solutions of the equation $A_i \cup B_j= [6]$ where $A_i \in \mathcal{A}_6$ and $B_j \in \mathcal{B}_6$. Although we mention that it is faster to verify that the complements of the sets in $(\mathcal{A}_6,\mathcal{B}_6)$ satisfy the complementary properties of recovering systems. 
\end{proof}

\begin{corr} \label{counteraharoni}
The pair $(\mathcal{A}^{'}_6,\mathcal{B}^{'}_6) := (\mathcal{A}_6 \cup \{ \emptyset \},\mathcal{B}_6 \cup \{ \emptyset \})$ is also a recovering pair, and it answers the question of Aharoni negatively since $$ \sum_{\substack{A_i \in \mathcal{A}^{'}_6 \\ B_j \in \mathcal{B}^{'}_6}}2^{|A_i \cap B_j|} = 67 > 64 = 2^6. $$
\end{corr}

Note that by blowing up the pair $(\mathcal{A}_6,\mathcal{B}_6)$ we get a lower bound on $f(n)$. Thus our knowledge about $f(n)$ can be summarized as  $$1.2009^ \leq 3^{1/6} \leq \lim_{n \rightarrow \infty} (f(n))^{1/n} \leq 1.3685. $$

\section{The new upper bound} 

\subsection{The combinatorial ideas}

Now we are aiming to improve the Körner-Holzman bound. We will often multiply a recovering pair with itself, thus let us denote the $r$-fold product of $(\mathcal{A},\mathcal{B})$ with itself by $(\mathcal{A}^r,\mathcal{B}^r)$. First we prove that subexponential factors can be ignored in the upper bounds of $|\mathcal{A}||\mathcal{B}|$.

\begin{claim} \label{tensorpower}
If we have that for every recovering pair $|\mathcal{A}||\mathcal{B}| \leq g(n)c^n$ for some $c>1$ and a fixed $g(n)$ such that $g(n)$ is subexponential $(\log(g(n))=o(n))$ then for every recovering pair $|\mathcal{A}||\mathcal{B}| \leq c^n$.
\end{claim}
\begin{proof}
Suppose that we have a recovering pair $(\mathcal{A},\mathcal{B})$ on a ground set of size $n_1$ such that $|\mathcal{A}||\mathcal{B}| > c^{n_1}$. Let $d$ be such that $|\mathcal{A}||\mathcal{B}| = d^{n_1}$. For a large enough $r$ we have a contradiction by $|\mathcal{A}^r || \mathcal{B}^r| = d^{n_1r} > g(n_1r)c^{n_1r}. $
\end{proof}

By the following lemma, we can assume some convenient properties of recovering pairs.

\begin{defi}
Let us call a recovering pair $(\mathcal{A},\mathcal{B})$ {\em uniform} if there exists a $k$ such that for any $A_i \in \mathcal{A}$ and $B_j \in \mathcal{B}$ we have that $|A_i|=|B_j|=k$, and {\em completely uniform} if it is uniform and $|\mathcal{A}|=|\mathcal{B}|$ also holds.
\end{defi}

\begin{lem} If there exists a $c>1$ such that for all $n$, we have that for any completely uniform recovering pair $(\mathcal{A}_u,\mathcal{B}_u)$ on a ground set of size $n$ we have that $|\mathcal{A}|| \mathcal{B}| \leq c^n$, then for any recovering pair on a ground set of size $n$ we have that  $|\mathcal{A}||\mathcal{B}| \leq c^n$. 
\end{lem}
\begin{proof}
For the sake of contradiction assume that we have a recovering pair $(\mathcal{A},\mathcal{B})$ on a ground set of size $n_1$ such that $|\mathcal{A}||\mathcal{B}| = d^{n_1} > c^{n_1}$. Let us remove all but the sets with the most frequent size among the elements of $\mathcal{A}^r$ and $\mathcal{B}^r$ to get $(\mathcal{A}_{f}^r,\mathcal{B}_{f}^r)$. By the pigeon hole principle, there must be at least $\mathcal{A}^r/(rn_1)$ sets with the most frequent size in $\mathcal{A}^r$, and similarly at least $\mathcal{B}^r/(rn_1)$ in $\mathcal{B}^r$. Now let us multiply $(\mathcal{A}^r_f,\mathcal{B}^r_f)$ with $(\mathcal{B}_f^r,\mathcal{A}_f^r)$ to get a completely uniform recovering pair $(\mathcal{A}_{g}^r,\mathcal{B}_{g}^r)$ on a ground set of size $2n_1r$ such that $$\frac{d^{2n_1r}}{(n_1r)^4} \leq | \mathcal{A}_g^r || \mathcal{B}_g^r | \leq c^{2n_1r}$$ which is a contradiction for large $r$.
\end{proof}

From now on we will assume that the recovering pair $(\mathcal{A},\mathcal{B})$ is completely uniform. To improve the Körner-Holzman bound, we will fine tune the approach in the previous chapter. We will introduce two parameters $u,t \in [0,1]$ of a recovering pair, that will control its size $|\mathcal{A}^r||\mathcal{B}^r|$. Thus knowing that the size is large, we will gain information about the parameters. Both $t$ and $u$ are functions of the recovering pair, but since it will not cause any confusion we always omit this dependence in the notation. 

\begin{defi} 
Let $u(r)$ be defined as follows. Take the size of every union $|A_i \cup B_j|$ such that $A_i \in \mathcal{A}^r$ and $B_j \in \mathcal{B}^r$. Let $u(r)$ be such that the number that is attained the most often (if there are more such numbers pick one arbitrarily) among these union sizes be equal to $u(r)nr$. Let $$u := \lim_{r \rightarrow \infty}u(r).$$
\end{defi}  

It is easy to see that $u(r)$ converges using Hoeffding's inequality \cite{hoeffding}. Let $X^r$ denote the probability distribution that takes two sets $A_i, B_j$ from $\mathcal{A}^r$ and $\mathcal{B}^r$ uniformly at random, and attains the value $|A_i \cup B_j|/(nr)$. We can think of $u(r)$ as the mode of $X^r$, and $u$ as the expected value of $X^r$ (or just the expected value of $X^1$, it does not depend on $r$).

\begin{defi}
Let $t(r)$ be defined as follows. Average the number of solutions of the equations $A_i \cup B_j=C$ for every set $C$ of size $u(r)nr$, where $A_i \in \mathcal{A}^r$ and $B_j \in \mathcal{B}^r$. Let $t(r)$ be such that this average be equal to $2^{t(r)u(r)nr}$. Formally

$$t(r):= \frac{1}{u(r)nr} \log_2 \left( \sum_{\substack{C \subset [n] \\ |C|=u(r)nr}} \frac{|\{(A_i,B_j)| A_i \cup B_j =C, A_i \in \mathcal{A}^r, B_j \in \mathcal{B}^r\}|}{\binom{nr}{u(r)nr}} \right) \quad  t := \lim_{r \rightarrow \infty}t(r) $$
\end{defi}

The limit exists, it is easy to see this using the bounds in the subsequent proof of Theorem \ref{equality}. The definitions are motivated by the following theorem.

\begin{thm} \label{equality}
If $(\mathcal{A},\mathcal{B})$ is a recovering pair on a ground set of size $n$, then $|\mathcal{A}||\mathcal{B}| = 2^{(h(u)+ut)n}$.
\end{thm}
\begin{proof}
$$|\mathcal{A}|^r|\mathcal{B}|^r = |\mathcal{A}^r||\mathcal{B}^r| = \sum_{\substack{A_i \in \mathcal{A}^r \\ B_j \in \mathcal{B}^r}} 1 = \sum_{C \subseteq [nr]}  |\{ (A_i,B_j) | A_i \in \mathcal{A}^r, B_j \in \mathcal{B}^r , A_i \cup B_j = C \}| \leq $$ $$ \leq  nr \sum_{\substack{C \subseteq [nr] \\ |C|=u(r)nr}} |\{ (A_i,B_j) | A_i \in \mathcal{A}^r, B_j \in \mathcal{B}^r , A_i \cup B_j = C \}| = nr \sum_{\substack{C \subseteq [nr] \\ |C|=u(r)nr} } 2^{t(r)u(r)nr} = $$ $$ = nr \binom{nr}{u(r)nr} 2^{t(r)u(r)nr} \leq nr 2^{(h(u(r))+t(r)u(r))nr} $$

Taking $r$-th roots and letting $r$ tend to infinity yields $|\mathcal{A}||\mathcal{B}| \leq 2^{(h(u)+tu)n}$. For the lower bound we work similarly.

$$|\mathcal{A}|^r|\mathcal{B}|^r = |\mathcal{A}^r||\mathcal{B}^r| = \sum_{\substack{A_i \in \mathcal{A}^r \\ B_j \in \mathcal{B}^r}} 1 = \sum_{C \subseteq [nr]}  |\{ (A_i,B_j) | A_i \in \mathcal{A}^r, B_j \in \mathcal{B}^r , A_i \cup B_j = C \}| \geq $$ $$ \geq   \sum_{\substack{C \subseteq [nr] \\ |C|=u(r)nr}} |\{ (A_i,B_j) | A_i \in \mathcal{A}^r, B_j \in \mathcal{B}^r , A_i \cup B_j = C \}| = \sum_{\substack{C \subseteq [nr] \\ |C|=u(r)nr} } 2^{t(r)u(r)nr} $$  $$ =  \binom{nr}{u(r)nr} 2^{t(r)u(r)nr} \geq  \frac{1}{u(r)nr} \sum_{i=0}^{u(r)nr}\binom{nr}{i} 2^{t(r)u(r)nr} \geq $$ $$ \geq \frac{1}{u(r)nr\sqrt{8nru(r)(1-u(r))}} 2^{(h(u(r))+t(r)u(r))nr}. $$
Again taking $r$-th roots and letting $r$ tend to infinity we established the lower bound and the proof is complete.
\end{proof}

Note that the main ideas behind this last proof are essentially the same as there in the proof of Theorem \ref{trivial}, but here we have more information about the recovering pairs with large $|\mathcal{A}||\mathcal{B}|$, in terms of $u$ and $t$. It is trivial that $u \in [0,1]$ and from the inequality established in Lemma \ref{legjobbfn} it follows that $t \in [0,0.4525]$. In this region, the function $2^{h(u)+tu}$ attains a single maximum, which is by no surprise $2.3685$, but now we know that there is a single choice of parameters $u$ and $t$ at which the function can attain this value. So if we manage to push these parameters away from this location, our upper bound on $|\mathcal{A}||\mathcal{B}|$ will improve. We will do this by introducing another upper bound for $|\mathcal{A}||\mathcal{B}|$ in terms of $u$ and $t$. 
The basic idea behind the following upper bound is that if for a fixed $C$ there are many solutions of the equation $A_i \cup B_j = C$, then among all $A_0 \in \mathcal{A}$ used in these solutions there must be many small differences of type $A_0 \setminus B$. But these must be different for different $A \in \mathcal{A}$, and there is not enough space for too many small $A \setminus B$ in $[n]$. Since our recovering pair is uniform, we can find small differences by finding large intersections.

\begin{defi}
Let $c$ be the relative size of the sets in $(\mathcal{A}, \mathcal{B})$, i.e. $c$ is such that $cn$ is the size of the sets in $\mathcal{A}$ (and also in $\mathcal{B}$ as the pair is completely uniform). The size of the sets in the pair $(\mathcal{A}^r,\mathcal{B}^r)$ is exactly $cnr$. Let us define $m_{Sr}:= 2c-u(r)$ and let the symmetric intersection size be defines as  $m_S = \lim_{r \rightarrow \infty}m_{Sr}.$
\end{defi}
Note that $m_{Sr}$ is the size of the intersection of every pair of sets $A_i,B_j$, for which $|A_i \cup B_j|=u(r)nr$. We call $m_{Sr}$ the symmetric intersection size, since for a fixed set $C_0$ of size $u(r)nr$, the solutions of the equation $A_i \sup B_j = C_0$ come in pairs $(A_1,B_1)$, $(A_2, B_2)$, $\ldots$ , $(A_w,B_w)$ and among the $w^2$ possible intersections, $m_{Sr}$ denotes the size of the ones where the indices are the same. These are the smallest intersections. We are aiming to find pairs with large intersections, but the size of the smallest ones will also play an important role. The next step will be to define an asymmetric intersection size which will be strictly larger than. To do this first we need a lemma that roughly states that a large enough proportion of the sets in $\mathcal{A}$ and in $\mathcal{B}$ is used as a solution of the equation $A_i \cup B_j = C_0$ where $C_0$ is such that there are a lot of solutions of this equation.
\begin{defi}
We say that a set $C \in [nr]$ of size $u(r)nr$ is crowded if there are at least $2^{t(r) u(r) nr-1}$ solutions to the equation $A_i \cup B_j = C$ in $(\mathcal{A}^r,\mathcal{B}^r)$. If for a set $A_1 \in \mathcal{A}^r$ there is a set $B_j \in \mathcal{B}^r$ such that $A_i \cup B_j = C$ we say that $A_i$ is used in $C$.
\end{defi}


\begin{lem} \label{vaneleg}
Among the sets $A_0 \in \mathcal{A}^r$ there are at least $|\mathcal{A}^r|/(2nr)$ ones, such that they are used in a crowded $C$.
\end{lem}
\begin{proof}
By the definition of $t(r)$, there are on average $2^{t(r) u(r) n}$ solutions for each set of size $u(r)nr$. Let us forget those $(A_i,B_j)$ pairs that have a union $C'$ of size $u(r)nr$ such that there are at most $2^{t(r) u(r) nr -1}$ solutions for $A'_i \cup B'_j=C'$. Then the average number of solutions can not decrease to less than $2^{t(r) u(r) nr-1}$. Thus at least half of the pairs which has their union of size $u(r)nr$ are used as a solution for a $C$ that has at least $2^{t(r) u(r) nr -1}$ solutions. Since by the definition of $u(r)$ and $t(r)$ we have that $$ \frac{1}{2nr}|\mathcal{A}^r||\mathcal{B}^r| \leq \binom{nr}{u(r)nr}2^{t(r)u(r)nr-1}  $$ at least $\frac{1}{2nr}|\mathcal{A}^r |$ of the sets in $\mathcal{A}^r$ have to be used to produce this many pairs. 
\end{proof}

Now we are ready to define the asymmetric intersection size. 

\begin{defi}
For each $A_j$ that is used in a crowded $C$, fix such a $C$ with solutions $(A_1,B_1), \ldots , (A_{2^{t(r)u(r)nr-1}},B_{2^{t(r)u(r)nr-1}}) \ldots$. Let $m_{A_jr}$ be such that the number that appears the most often among the numbers  $|A_j \cap B_1|, \ldots , |A_j \cap B_{2^{t(r)u(r)n}}|$ be equal to $m_{A_jr}nr$ (if there are more such numbers, pick one arbitrarily). Let $m_{Ar}$ be the number that appears the most often among the numbers $m_{A_jr}$ where $A_j$ is used in a crowded $C$ (if there are more such numbers, pick one arbitrarily). Finally we define the asymmetric intersection size  $m_A$ as  $$m_A := \liminf_{r \rightarrow \infty} m_{Ar}.$$
\end{defi}

Note that we do not know anything about the convergence of $m_{Ar}$. Our subsequent arguments work if we choose any accumulation point of the sequence $m_{Ar}$ instead of the smallest one. Later we will prove lower bounds of $m_S$ and $m_A$, but we will not need them before we are trying to quantify our results. We would like to find large intersections, to have small differences. Our last lemma before the proof of the second upper bound roughly says that for a set $A_0 \in \mathcal{A}$ we not only have intersections of size $m_A$, but there are exponentially many different such intersections, forcing exponentially many different differences.  



\begin{lem} \label{sokkul}
Let $A_1 \in \mathcal{A}^r$ be such that it is used as a solution in a crowded $C$. Then there are at least $2^{(t(r)u(r)-m_{A_1r}+m_{Sr})nr-1}$ different intersections of size $m_{A_1r}$ of type $A_1 \cap B$. 
\end{lem}
\begin{proof}

There are $2^{t(r)u(r)nr-1}$ pairs $(A_i,B_i)$ such that their union is $C$, we will use only these sets. Among these $B_i$, let $B_{i_1} , \ldots , B_{i_K}$ denote those that have the same intersection $I$ of size $m_{A_1r}$ with the set $A_1$. Consider the recovering pair that consists of these $B_{i_k}$, and the corresponding $A_{i_k}$ for which $B_{i_k} \cup A_{i_k} = C$. We claim that the system 
$$\mathcal{A}':= \{A_{i_k} \setminus ( A_1 \setminus I ) | k \in [K] \} \quad \mathcal{B}' := \{B_{i_k} | k \in [K] \} $$
is a recovering pair, on $(u(r)-c+m_{A_1r})nr$ elements, since the set $A_1 \setminus I$ is disjoint from all $B_{i_k}$ and is contained by all $A_{i_k}$. Every $A_{i_k}$ must be different on the complement of $B_{i_1}$, thus  $K \leq 2^{(u(r)-c+m_{A_1r}-c)nr}= 2^{(m_{A_1r}-m_{Sr})nr}$ and the proof is complete.

\end{proof}

\begin{thm} \label{pushing} 
If $(\mathcal{A},\mathcal{B})$ is a completely uniform recovering pair, then $$|\mathcal{A}||\mathcal{B}| \leq \binom{n}{(c-m_A)n}^2 2^{2(m_A-tu-m_S)n} \leq 2^{2(h(c-m_A)+m_A-tu-m_S)}.$$
\end{thm}
\begin{proof}
Since the pair is completely uniform it is enough to bound $|\mathcal{A}|$. By Lemma \ref{vaneleg} we have that
$$ |\mathcal{A}|^r= |\mathcal{A}^r| \leq 2nr|\{A_i | A_i \in \mathcal{A} , A_i \text{is used in a crowded set}\}| \leq $$
$$ \leq 2n^2r^2|\{A_i | A_i \in \mathcal{A} , A_i \text{is used in a crowded set}, m_{A_ir}=m_{Ar} \}|.  $$
By Lemma \ref{sokkul}, every such $A_i$ has at least $2^{(t(r)u(r)-m_{Ar}+m_{Sr})nr-1}$ different intersections of size $m_{Ar}nr$. Thus every such $A_i$ has at least $2^{(t(r)u(r)-m_{Ar}+m_{Sr})nr-1}$ different differences of size $(c-m_{Ar})nr$, which must be different by the definition of a recovering pair. Thus we have that 
$$  2n^2r^2|\{A_i | A_i \in \mathcal{A} , A_i \text{is used in a crowded set}, m_{A_ir}=m_{Ar} \}| \leq $$ 
$$ \leq 2n^2r^2 \binom{nr}{(c-m_{Ar})nr} 2^{-(t(r)u(r)+m_{Ar}-m_{Sr})nr+1}.$$
Taking $r$-th roots and letting $r$ tend to infinity finishes the proof.

\end{proof}

Note that the proof of Theorem \ref{pushing} is rather straightforward once we have the appropriate definitions. Intuitively speaking we feel that the whole argument is about that we would like to have a large subset of sets in $\mathcal{A}$ such that they have small differences with some set from $\mathcal{B}$, and since all these must be different, there can not be too many of them, as there is not enough space. It is not the proof of Theorem \ref{pushing} that is important, but the fact that we can prove effective bounds on the parameters used there. To show that Theorem \ref{equality} and Theorem \ref{pushing} together improve on the Körner-Holzman bound, we will present lower bounds on $m_S$ and $m_A$ in the next section.

\section{Quantifying the upper bound}

To improve the Körner-Holzman bound, we need the following lower bounds on $m_S$ and $m_A$. 

\begin{lem} \label{msbecsles}
$h_i^{-1}(t) u \leq m_S.$ 
\end{lem}
\begin{proof}
Since the average number of solutions of $A_i \cup B_j= C$ where $A_i, B_j$ are from $\mathcal{A},\mathcal{B}$ and $|C|=u(r)n$ is $2^{t(u)u(r)n}$,  there must be at least one set $C'$ such that there are at least $2^{t(u)u(r)n}$ solutions. Let us fix such a $C'$. By Lemma \ref{cupcap}, for each solution $(A_i,B_i)$ of the equation $A_i \cup B_j= C'$, we have that $A_i \cap B_i$ is unique. Thus 
$$ 2^{t(u)u(r)n} \leq \binom{u(r)n}{m_{Sr}n} = \binom{u(r) n}{\frac{m_{Sr}}{u(r)}u(r)n} \leq 2^{h\left(\frac{m_{Sr}}{u(r)}\right)u(r)n}  $$
$$ t(u)u(r) \leq h \left(\frac{m_{Sr}}{u(r)}\right)u(r) $$
$$ h_i^{-1}(t(u))u(r) \leq m_{Sr} \leq h_d^{-1}(t(u))u(r) $$
and by taking $r$ to infinity we obtain 
$$ h_i^{-1}(t)u \leq m_{S} $$
and the proof is complete.
\end{proof}

\begin{lem} \label{mabecsles}
$$ c-h_d^{-1} \left( \frac{tu}{u-c} \right) (u-c) \leq m_A $$
\end{lem}
\begin{proof}
Fix a set $A_j$ that is used in a crowded $C$. Since all the sets $B_1, \ldots , B_{2^{t(r)u(r)nr-1}}$ that are used in $C$ are different outside of $A_j$, they must differ on a set of size $(u(r)-c)nr$. The most frequent size among $ |B_1 \setminus A_j|, \ldots , |B_{2^{t(r)u(r)nr-1}} \setminus A_j|$ is by the definition of $m_{A_jr}$ equal to $(c-m_{A_jr})nr$. Thus 

$$2^{t(r)u(r)nr-1} \leq  (u(r)-c)nr \binom{(u(r)-c)nr}{(c-m_{A_jr})nr} \leq (u(r)-c)nr 2^{h\left( \frac{c-m_{A_jr}}{u(r)-c} \right) (u(r)-c)nr } $$

for all $r$ and all $A_j$ thus 

$$tu \leq h\left(  \frac{c-m_{A}}{u-c} \right)(u-c) $$
holds. After a simple rearrangement, the statement of the lemma follows:
$$h_i^{-1} \left( \frac{tu}{u-c} \right) (u-c) \leq c-m_A \leq h_d^{-1} \left( \frac{tu}{u-c} \right) (u-c)$$
$$ c-h_d^{-1} \left( \frac{tu}{u-c} \right) (u-c) \leq m_A $$
\end{proof}

Now we proceed with the proof of the following claim, that finishes our proof, that for any recovering pair  $|\mathcal{A}||\mathcal{B}| \leq 2.284^n$.

\begin{claim}
For any $u \in [0,1]$ and $t \in [0,0.4525]$ we have that $$ \min \{ 2^{(h(u)+tu)n}, 2^{2(h(c-m_A)+m_A-tu-m_S)n}  \} \leq 2.284^n $$ or equivalently
$$ h(u)+tu \leq 1.1922 \quad \text{or} \quad h(c-m_A)+m_A-tu-m_S \leq 0.5961 . $$
\end{claim}

Note that the bound $2.284$ can be improved to $2.2815$ by using more advanced computer calculations. Here we only present a proof of $2.284$ which uses a computer only to evaluate a function. 

\begin{claim} \label{easyanalysis}
For any a fixed $t \in [0,0.4525]$ we have that for $u \leq (1+2^{-t})^{-1}$ the function $h(u)+tu$ is monotone increasing in $u$ and for $u \geq (1+2^{-t})^{-1}$ is is monotone decreasing in $u$.
\end{claim}
\begin{proof}
The derivative of $h(x)$ is $- \log_2 (\frac{x}{1-x})$. For fixed $t$ the function $h(u)+tu$ is unimodular, since the derivative of $h(u)$ is monotone decreasing in the interval $u \in [0,1]$. For fixed $t$, the maximum of $h(u)+tu$ is at $u = (1+2^{-t})^{-1}$.
\end{proof}

We will call $h(u)+tu$ the first and $h(c-m_A)+m_A-tu-m_S$ the second bound. Since the first bound is a function of two variables, and the second is a function of four (since $c=(u+m_S)/2$) we are aiming to eliminate $m_A$ and $m_S$ from the second bound. Then we will establish certain monotonicity properties of these bounds, such that evaluating them on $16$ places will yield our claim. Note that evaluating them on more places would yield better bounds, but the improvement is in the third decimal digit.
We start our work with narrowing the range of parameters using the first bound.

\begin{claim}
Outside of the rectangle $u \in [0.4400,0.7100] \text{and} \, t \in [0.3600,0.4525]$ we have that $h(u)+tu \leq 0.5961$.
\end{claim}
\begin{proof}
The function $h(u)+tu$ is trivially increasing in $t$ and unimodular in $u$. For $t=0.36$ by Claim \ref{easyanalysis} its maximum (as a function of $u$) is attained at $1/(1+2^{-0.3600})$ and it is less than $1.1922$ so when $t \leq 0.36$ it is smaller than $1.1922$ for every value of $u$. For $t=0.4525$ it is less than $1.1922$ outside of the interval $u \in [0.4400,0.7100]$ and we are done by unimodularity in $u$ and monotonicity in $t$. 
\end{proof}

From now on we will assume that $u \in [0.4400,0.7100] $ and $ \, t \in [0.3600,0.4525]$. The following lemma will be useful in the proofs that $m_A$ and $m_S$ should be minimized. 

\begin{lem} \label{furcsa}
$$ A(u,t,m_S):=  h^{-1}_d \left( \frac{2tu}{u-m_S} \right)\frac{u-m_S}{2} < 1/3$$
\end{lem}
\begin{proof}
Since $h^{-1}_d(x)$ is decreasing, $A(u,t,m_S)$ is trivially monotone decreasing in $m_S$ so after using $h^{-1}_i(t)u \leq m_S$ we get
$$ A(u,t,m_S) \leq h^{-1}_d \left( \frac{2t}{1-h^{-1}_i(t)} \right)\frac{u(1-h^{-1}_i(t))}{2} $$
which is trivially monotone decreasing in $t$ and monotone increasing in $u$ so substituting $u=0.7100$ and $t=0.3600$ and checking that 
$$ h^{-1}_d \left( \frac{0.72}{1-h^{-1}_i(0.36)} \right)\frac{0.71(1-h^{-1}_i(0.36))}{2} < \frac{1}{3} $$
finishes the proof. 
\end{proof}

We proceed with showing that in the second bound $m_A$ should be minimized. Note that for this it is 

\begin{claim}
For fixed $u \in [0.44,0.71]$ and $t\in [0.36,0.4525]$ and $m_S$, the bound $h(c-m_A)+m_A-tu-m_S$ is maximal if $m_A$ is minimal. 
\end{claim} 
\begin{proof}
To maximize $h(c-m_A)+m_A-tu-m_S$ we have to minimize $m_A$ if and only if its derivative with respect to $m_A$ is negative. Thus we need that 
$$-h'(c-m_A)+1 \leq 0 $$
which, since $h'(x)$ is decreasing and it attains one at $x=1/3$, is equivalent to 
$$c-m_A \leq 1/3. $$
By rearranging the statement of Lemma \ref{mabecsles} we have that 
$$c-m_A \leq h^{-1}_d \left( \frac{tu}{u-c} \right)(u-c)= $$
which we can rewrite using $u-c=(u-m_S)/2$ to get 
$$= h^{-1}_d \left( \frac{2tu}{u-m_S} \right)\frac{u-m_S}{2}$$
which is smaller than $1/3$ by Lemma \ref{furcsa}.
\end{proof}

Thus the second bound became 
$$ h\left(h_d^{-1} \left( \frac{tu}{u-c} \right) (u-c)\right)+c-h_d^{-1} \left( \frac{tu}{u-c} \right) (u-c)-tu-m_S =$$
$$ h \left( h_d^{-1} \left( \frac{2tu}{u-m_S} \right) \frac{u-m_S}{2} \right)+\frac{u-m_S}{2}-h_d^{-1} \left( \frac{2tu}{u-m_s} \right) \frac{u-m_S}{2}-tu =$$
$$ h(A(u,t,m_S))-A(u,t,m_S) +\frac{u-m_S}{2}-tu.$$
Now we show that $m_S$ should be minimized. 

\begin{claim} \label{asdasd}
For fixed $u,t$ the bound $ h(A(u,t,m_S))-A(u,t,m_S) +\frac{u-m_S}{2}-tu$ is maximal if $m_S$ is minimal. 
\end{claim}
\begin{proof}
The $\frac{u-m_S}{2}-tu$ part is monotone decreasing in $m_S$, for the $h(A(u,t,m_S))-A(u,t,m_S)$ part observe that $A(u,t,m_S)$ is monotone decreasing in $m_S$ and the function $h(x)-x$ is increasing if $x \leq 1/3$ which inequality is guaranteed by Lemma \ref{furcsa}.
\end{proof}

So the second bound takes the form
$$ h \left( h_d^{-1} \left( \frac{2t}{1-h_i^{-1}(t)} \right) \frac{u}{2}(1-h_i^{-1}(t)) \right)-h_d^{-1} \left( \frac{2t}{1-h_i^{-1}(t)} \right) \frac{u}{2}(1-h_i^{-1}(t))+\frac{u}{2}(1-h_i^{-1}(t)-2t). $$
From the proof of Claim \ref{asdasd} we see that the second bound is monotone decreasing in $t$, and is monotone increasing in $u$ since the positivity of $1-h_i^{-1}(t)-2t$ is guaranteed by Lemma \ref{legjobbfn}. 

Note that $h(u)+tu$ is trivially increasing in $t$. By the monotonicity properties of our bounds, evaluating the first and the double of the second bound alternately (starting with the second) in the points below, we can deduce that for any $u \in [0.4400,0.7100] $ and $ t\in [0.3600,0.4525]$ one of the bounds is smaller than $2.284$. Note that if we connect the points, the resulting shape resembles a staircase. The points $(u_i,t_i)$ are: $(0.5893, 0.36),$ 
$ (0.5893, 0.364) , (0.599, 0.364) , (0.599, 0.367), (0.607, 0.367) , (0.607, 0.37) , (0.615, 0.37), $ \\
$ (0.615, 0.374) , (0.627, 0.374) , (0.627, 0.38), (0.645, 0.38),(0.645, 0.392), (0.688, 0.392), $ \\
$(0.688, 0.43),(0.71, 0.43),(0.71, 0.4525).$

\section{Concluding remarks and open problems}

Note that in the last section we proved that for a certain range of parameters, $|\mathcal{A}||\mathcal{B}| < 2^{(h(u)+tu)n}$ (where the second bound is stronger than the first one) which contradicts Theorem \ref{equality} showing that this range of parameters is impossible to achieve by a recovering pair. Note that there are completely uniform recovering pairs with parameters $u \in [0,1]$ and $t=0$, but we do not know whether $t>0$ is possible at all. From a proof that $t>0$ is impossible, or a proof that presents an upper bound that contradicts Theorem \ref{equality} for a large enough range of parameters $u$ and $t$, would follow Simonyi's conjecture. A less ambitious way of improving our results would be to improve the upper bound on $f(n)$.

\begin{question}
$$\lim_{n \rightarrow \infty}(f(n))^{1/n}=?  $$
\end{question} 

We know that this quantity is in the interval $[1.2009,1.3685]$, we call this less ambitious, as the lower bound shows that this method can not prove Simonyi's conjecture without additional ideas. Intuitively speaking, $f(n)$ prevents the concentration of the unions on a single set. It would be interesting to have a lemma that prevents the concentration of the unions on some sets close to each other. Another interesting way to prevent the concentration of sets is to punish large intersections, the sum proposed by Aharoni does exactly this. Its asymptotic exponent would also be of independent interest. 

\begin{question}
What is the maximal asymptotic exponent of the sum 
$$ \sum_{\substack{A_i \in \mathcal{A} \\ B_j \in \mathcal{B}}}2^{|A_i \cap B_j|} $$
where $(\mathcal{A}, \mathcal{B})$ is a recovering pair?
\end{question}

We know that this quantity is somewhere in the interval $[2.0153,3]$. Let us finish with a question that is more of an invitation to produce a "counterexample". Can $t$ be larger than zero? From a negative answer to this question would immediately follow Conjecture \ref{Simonyi}, and a positive answer would be interesting too, since in any counterexample to Conjecture \ref{Simonyi}, $t$ is necessarily positive. 

\section{Acknowledgements}

The author would like to thank Gábor Simonyi, Márton Zubor and Ron Aharoni for stimulating discussions.

\end{document}